\documentclass[12pt]{article}
\usepackage{amssymb, amsfonts, amsmath}
\usepackage{amsthm}
\usepackage[utf8]{inputenc}
\usepackage{amssymb, amsfonts, amsmath}
\usepackage{amsthm}

\newtheorem{theorem}{Theorem}
\newtheorem{definition}{Definition}
\newtheorem{lemma}{Lemma}

\usepackage[mathcal]{eucal}
\usepackage{mathrsfs}

\begin{document}

\title{Solution of the word problem for semigroups without cycles}

\author{A. Malkhasyan}
\date{}

\maketitle

\paragraph{Introduction}

This article provides a solution to the problem of word equality in so called semigroups without cycles. In 1966 Sergei Ivanovich Adian published a monograph in
Proceedings of the Steklov Institute of Mathematics, which particularly proved
that the problem of word equality is decidable in one-relator semigroups that have the cancellative property. 
 
Wherein, S.I.Adian introduced the concept of semigroups without cycles, for which the theorem on its embeddability in a group with the same relations was proved. Since a one-relator cancellative semigroup is a semigroup without cycles, it allows using Magnus’s algorithm which makes the problem of word equality decidable in one-relator groups (see \cite{2}, \cite{3}).

For further presentation, it would be appropriate to recall some concepts and facts from \cite{1}, \cite{2}.
Let the semigroup $\Pi$ be given by the generators $$a_{1}, a_{2}, \ldots, a_{q}$$ and defining relations  $$P_{i}=Q_{i},  i=1,2, \ldots, {l}.$$ 
To each generator $a$, that is the beginning of some defining word, there is assigned a one-to-one point on the plane, which we will call the vertex $a$.

Two vertices are connected by an edge if and only if they come first, which means they are the leftmost letters of the parts of some defining relation of the semigroup $\Pi$. The graph that is constructed this way is called the left graph of the semigroup $\Pi$.

The right graph of the semigroup $\Pi$ is constructed similarly if we consider the ends of the defining relations instead of the first letters of the defining words.
The semigroup $\Pi$ is called non-cyclic if the left and the right graphs do not contain cycles.

In \cite{1}, where the aforesaid Adjan's theorem was proved, about the semigroup $\Pi$ embeddability in a group with the same relations, the decisive result was the word $aX$ ($Xa$) is equal to the word $aY$ ($Ya$)  in semigroup $\Pi$, if and only if the word $X$ is equal to the word $Y$. Based on this fact some unambiguous representations of the words of semigroup $\Pi$ were constructed.

Let us briefly outline the essence of these constructions, in view of their frequent use when presenting the results of this paper. If the word $U$ is empty, then it does not have a prefix presentation. Let $U$ be the word of the semigroup $\Pi$, the first letter of which is $a$ and it is required to find out whether it is equal to the word $V$, with the first letter $b$, which we will call a guide letter for $U$.

It is clear that equality is possible only if in the left graph of the semigroup $\Pi$ there is a path, beginning from the vertex $a$ and ending in the vertex $b$. In \cite{1} it is shown that a pair of letters $a$,$b$ defines some unambiguous representation of the word $U$, which we call its prefix presentation.
 
Since the graph of the semigroup $\Pi$ has no cycles, then there is only one path connecting vertex $a$ to(with) vertex $b$. Let $a$, $a_{i_{1}}$,\ldots, $a_{i_{t}}$, $b$ - consecutive vertices of this path, and, therefore, there is only one defining relation in the form $P_{i}=Q_{i}$, first letters of which, respectively, $a$ and $a_{i_{1}}$. (If there is no such defining word, it will mean that the prefix presentation of the word $U$ does not exist).

By $\alpha_{1}$ we denote the longest common beginning of the words $ U $ and $ P_ {i} $. Then $U\eqcirc \alpha_{1} U_1$, $P_{i} \eqcirc \alpha_{1} P'_{i}$ and if the word  $\alpha_{1}$ will be the defining word $P_{i}$, then the construction of the word representation $U$ is completed and we will call it prefix representation, and the the word $P_{i}$ itself is the prefix head of that representation. If the length of $P_{i}$ turns out to be greater than the length of its beginning $\alpha_{1}$, assuming that the prefix representation of the word $U$-$PR(U)\eqcirc \alpha_{1} PR(U_1)$, the first letter $a_{j}$ of the word $U_1$ should transform to the first letter $c$ of some defining word $P'_{i}$, thereby uniquely defining the guide letter for the word $U_1$. 
As a result of this construction, the word $U$ can be uniquely represented in the form

$$U \eqcirc \alpha_{1} \ldots \alpha_{k} U_{k}, (k \geq 0)$$. 

Note that, if the letter $a$ coincides with $b$ at the beginning of the process, then as $\alpha_ {1}$ we take the letter $a$, and we continue our construction for the letters following $a$ in the words $U$ and $V$.

 The prefix representation of the word $U$ is considered constructed, if either $\alpha_{k}$ is a prefix head or the first letter $a_{i_ {k}}$ of the word $U_ {k}$ is not the first letter of any of parts of a suitable constitutive defining relation in the form $P=Q$.
 This will mean that in the graph of the semigroup $\Pi$ there is no suitable path leading from the vertex $a$ to the vertex $b$. In particular, this means that the word $U$ is not equal to the word $V,$ and the word $U$ will be represented as a sequence of prefixes without a head.

The suffix representation of the word $U$ is defined analogically if $c$ and $d$ are considered last letters of $U$, $V$ respectively. Thus, $U$ will have the following representation $U \eqcirc U_{m} \beta_{m} \ldots \beta_{1}$ where $\beta_{m}$ is either the suffix head of the word U or the suffix head is missing in the representation of the word.
The words $\alpha_{i}$ $(\beta_{i})$ we will call prefixes(suffixes) of the representation of the word $U.$

Thus, an arbitrary word $U$ can be represented with respect to $(b,d)$ as first and last letters of words $U$ and $V$ respectively, both in the prefix and suffix forms. More about the uniqueness of these presentations can be found in (\cite{1}, \cite{2}).

We’ll call the transformation of word $U$ to $V$ a prefix transformation if in the word $U$ written in its prefix representation, the head $P_{j}$ is replaced with an unambiguously defined word $Q_{j}$, in other words, a simple transformation of a semigroup $\Pi$ occurs, after which the obtained word $W$  is written in its prefix representation.(The suffix representation transformation  of the word $U$ is defined analogically). 
The word $U_1$, that is obtained from the word $U$ as a result of that transformation, will be called the prefix transformation form of $U$ with the first rank. Furthermore, we’ll call $U_{k}$ a prefix form of $U$ with $k$ rank, if it is a one-step transformation of $U_{k-1}$ with rank {k-1}.(The suffix transformation form is defined analogically).

\begin{lemma}\label{ekvivalent} An arbitrary prefix form $U_{k}$ of the word $U$ from the sequence of the transformation of the word $U$ to $V$, is also  a suffix form of the word $U$ with the same rank.
 \end{lemma} 

\begin{proof} 
	Let \begin{equation} \label{psf} U=U_1 \rightarrow U_2 \rightarrow \ldots \rightarrow U_{\lambda} \end{equation}
be a sequence of prefix transformations, meaning each term of the sequence is a prefix form of the word $U.$ We’ll prove by induction on $\lambda$ rank of the prefix form. If  $\lambda=1$ then the prefix head of the word $U$ coincides with its suffix head. Let $\lambda>1$ and $U_{1}, U_{2,} \ldots, U_{\lambda}$ is the sequence of prefix forms of the word $U$.

It’s clear that in the sequence $$U \rightarrow U_1 \rightarrow U_2 \rightarrow \ldots \rightarrow U_{\lambda},$$ each $U_{j}$ is a prefix form of the word $U_{i},$ where ${i<j}.$
Let $$U_{l} \rightarrow U_{1+1}, 0 \leq l < \lambda,$$ be the first prefix transformation of this sequence that has the following form $$U_{l} \rightarrow U_{1+1}, l < {\lambda},$$ $$U_{1} \eqcirc K Q L,  U_{1+1} \eqcirc K R L.$$
The word $L$ is not changed during transformations, so this transformation is changing the leftmost letter of $U$. 

Let’s look into the subsequence $U_1 \rightarrow \ldots \rightarrow U_{t}$ in the sequence \eqref{psf}, where $Q$ is introduced for the first time in the word $U_{t}$, $U_{t} \eqcirc K` Q L,$  $0 \leqslant t \leqslant l.$ 

Now, first, using the induction, we’ll replace the considered subsequence with the  subsequence of suffix transformations. Then we’ll transform the word $U_{t}$ to $U_{t+1} \eqcirc K` P L,$, clearly with suffix transformation. After this we’ll replace the subsequence $U_{t} \rightarrow \ldots \rightarrow U_{1}$ with the subsequence $V_{t+1} \rightarrow \ldots \rightarrow V_{1+1},$ where $V_{i+1}$ is obtained from word $U_{i}$ with the replacement of the indicated occurrence $Q$ by $P,$ and as in that subsequence the subword $PL$ is not affected, we’ll use induction to change it to a sequence of suffix transformations $U_{l+1} \rightarrow \ldots \rightarrow U_{\lambda}$. By the principle of mathematical induction, the subsequence of prefix transformations can be replaced by the subsequence of suffix transformations.
 \end{proof}

It’s obvious that if $U$ and $V$ are equal in the semigroup $\Pi$, then there exists a sequence of prefix(suffix) transformations, changing the word $U$ into the word $v$, that is, $V$ will be the prefix form of the word $U$. It is clear that, in the end, the letter $a$ will transform into the letter $b$ during transformations, or the sequence will be interrupted due to the absence of the head in one of the $U_{j}$ words, and finally, the sequence can turn to be infinite.
Let’s look into the properties of infinite sequences of prefix(suffix) transformations with the following form \begin{equation} \label{pref.preob} U \eqcirc U_1 \rightarrow U_2 \rightarrow \ldots \rightarrow U_{n} \rightarrow\ldots \end{equation} 

\begin{definition} \label{nepod}
The pair of subwords $K$ and $L$, the word $U_{n} \eqcirc K Q L$, the sequences in the form of \eqref{pref.preob}, that are not a part of further transformations $$U_{n} \rightarrow\ldots,$$ will be called the span of the word $U_{n}$,with the ordered pair of edges $(K, L),$ and the subword $Q,$ the first and last letter of which have been affected by those transformations at least once,  will be called the kernel.
\end{definition}

From infinity of the given sequence and due to the  fact that the  defining relations of the semigroup do not contain cycles for any of $U_{n}$ terms from the sequence \eqref{pref.preob}, in the form of $K Q L$ there exists a number $m$,  ($m$ > $n$) so that in the infinite  sequence $U_{m} \rightarrow \ldots$, the letter after the subword $K$ and the letter before the subword $L$, are not affected by the above transformations. In other words, the terms of the sequence \eqref{pref.preob} will have expanding span, herewith, it’s obvious that the span will be words consisting of prefixes( $K$ left edge) and suffixes( $L$ right edge).

 \begin{lemma}   \label{raspolgol}

Any word $U_{s}$ of the sequence\eqref{pref.preob}, or, in other words, any prefix form of the word $U$ has the form $$U_{s} \eqcirc K \alpha_{1} \ldots \alpha_{n} A X B \beta_{1} \ldots \beta_{m} L,$$ where $A$ is the prefix, and $B$ is the suffix head of the word $U_{n}$, or $$U_{s} \eqcirc K\alpha_{1} \ldots \alpha_{n} A  \beta_{1} \ldots \beta_{m} L,$$ where $A$ is the prefix and the suffix head at the same time, ($(K,L)$ is the span of the word $U_{n}$).
\end{lemma}

 \begin{proof}
 Let us show that other arrangements of prefix and suffix heads in the words $U_{s}$ are impossible.
 First, note that the suffix head of the word $U_{s}$ cannot cut any prefix of that word. Indeed, if  $$U_{s} \eqcirc K \alpha_{1} \ldots\alpha'_{t} S \beta_{i} \ldots\beta_{1} L,$$ where $S$ is the suffix head, the beginning of which coincides with the beginning of the subword $\alpha''_{t}$, where $\alpha'_{t}\alpha''_{t} \eqcirc \alpha_{t}$is the prefix, then the suffix transformation, that replaces the defining word $S$ with some defining word  $Q_{j},$ changes the first letter of the subword $\alpha''_{t},$ but then the defining relation that changes the beginning of the word $\alpha'_{t}$ cannot be found without returning $S.$ However, this can happen if an ongoing sequence of suffix transformations translates the word $S \beta_{i} \ldots\beta_{1}L$ into a word in the form $S B L,$ which contradicts Adjan theorem. 
If $\alpha''_{t}$ is empty, then $S$ coincides with the prefix head of the word $U_{n}.$ Cases  
$$U_{n} \eqcirc K \alpha_{1} \ldots\alpha'_{t} S_{1} S S_{2} \beta_{i} \ldots\beta_{1} L,$$  
where $S_{1} S$ is a prefix, and $S S_{2}$ is a suffix head, or $S_{1} S S_{2}$ is a prefix head, and $S$-suffix heads are brought to the symmetrical reasoning discussed above, since in these cases the suffix head cuts the prefix head.
 \end{proof}

Let \eqref{pref.preob} be a sequence of prefix transformations, words $U$ to word $V$, $a$ is the first letter of a word $U$, $b$ is the guide letter of this sequence. If the first letter of any of the terms of the sequence becomes $b$ the sequence ends there, and then we can consider the transformations of the obtained prefix form $U$
into  the word obtained from $V$ by erasing the first letter $b.$  
We will focus on the case when sequence \eqref{pref.preob} is infinite: none of the forms of the word $U$ begins with the letter $b.$  
 
By virtue of Lemma \eqref{ekvivalent} every two terms of this sequence $U_{s}$ and $U_{t}$, being equal in the semigroup $\Pi$, are translated from one to the other by both prefix and suffix transformations. Since the sequence of transformations is infinite, and the defining relations of the semigroup $\Pi$ do not contain cycles, as we have already noted the first (and last) letters of the terms of the sequence  \eqref{pref.preob}), starting from a certain number,will not participate in further transformations.  

We will use the record of the sequence terms in the forms indicated in \eqref{raspolgol}. Further, the subword  $A X B \beta_{1} \ldots AXB\beta_{m} L$ is the end of the representation of the word $U_{n}$ and the subword $K \alpha_{1} \ldots \alpha_{n} A X B$ is its beginning.

\begin{lemma}   \label{ekvivalent1}
Let the $U_{p}$ and $U_{q}$ be the words of the alphabet of the semigroup $\Pi$, terms of the sequence \eqref{pref.preob} in the form 
$$U_{p} \eqcirc K \alpha_{1} \ldots \alpha_{t } RXS \beta_{m} \ldots \beta_{1} L$$
$$U_{q} \eqcirc K' \gamma_{1} \ldots \gamma_{l}RXS\beta_{m} \ldots \beta_{1} L'$$ 
with the spans $(K, L)$ and $(K', L')$, respectively.
 
 Let the words $U_{p}$ and $U_{q}$ have the same suffix representations, $R$ and $S$, respectively, are the suffix and prefix heads of those words. Then the words that will be obtained after a single use of the suffix transformation to these words, will have graphically equal ends, with the same suffix.
\end{lemma}
 \begin{proof}
  
The statement is obvious if $R$ does not coincide with $S$, since the transformation either occurs to
the right of the word $R$ in the graphically equal words $X$ $S$ $\beta_{m} \ldots \beta_{1} L$ with the same suffix representations, or after the suffix transformation of the word $R$ to the corresponding defining word $D$, the suffix head will move to the word $R.$ There can be no other cases due to Lemma \eqref{raspolgol} and the infinity of the sequence \eqref{pref.preob}

Let us suppose that
$$U_{p} \eqcirc K \alpha_{1} \ldots \alpha_{t } R \beta_{m} \ldots \beta_{1} L$$
$$U_{q} \eqcirc K' \gamma_{1} \ldots \gamma_{l} R \beta_{m} \ldots \beta_{1} L'$$

For any suffix transformation specified by the transition of $R$ to $Q$ in the words
$$U_{p+1} \eqcirc K \alpha_{1} \ldots \alpha_{t } Q \beta_{m} \ldots \beta_{1} L$$ and
$$U_{q+1} \eqcirc K' \gamma_{1} \ldots \gamma_{l} Q \beta_{m} \ldots \beta_{1} L'$$  
and the prefix and suffix heads must affect the word $Q$. Indeed, if the prefix head, especially the suffix head, turns out to be to the left of the word $Q,$ then this would mean that it would be a head in the words $ U_ {p} $ and $ U_ {q} $, contrary to the assumption that it is $ R. $. The possibility of finding a prefix or suffix head to the right of $ Q$ is also refuted.Consider the possible positions of prefix and suffix heads in the words $U_{p+1}$ and $U_{q+1}.$

Case 1.
$$U_{p+1} \eqcirc K \alpha_{1} \ldots \alpha_{t} S_1 X D_1 \beta_{m} \ldots \beta_{1} L,$$
$$U_{q+1} \eqcirc K' \gamma_{1} \ldots \gamma_{l} S`1 X` D_1 \beta{m} \ldots \beta_{1} L',$$ 
where $Q \eqcirc S_1 X D_1 \eqcirc S`1 X` D_1$, $D \eqcirc D_1 \beta{m}$ is the suffix head in both words. If $\alpha_{t} S_1$ is the prefix head in the word $U_{p+1}$, considering that the sequence of transformations is infinite \eqref{pref.preob}, and using Lemma \eqref{ekvivalent}, there is a number $k$, so that as a result of suffix transformations in the word $V_{p+k+1},$ the suffix head obtained from $U_{p+1}$ coincides with the prefix head $\alpha_{t} S_1.$
Wherein, all these transformations are done in the subword $S_1 X D_1 \beta_{m} \ldots \beta_{1},$ and the suffix representation of the subword obtained from it has the form $S_1 \theta_r \ldots  \theta_1$ in the word
 $V_{p+k}$.
If these same transformations are applied to the word $U_{q+1},$ it is easy to notice that they follow the aforementioned suffix transformations, and therefore,
the subword $S'1 X' D_1 \beta{m} \ldots \beta_{1}$ transforms into the same subword $S_1 \theta_r \ldots  \theta_1$ with the same suffix representation. Finally, since $S_1$ is the ending of the suffix head with the same guide letter in both words, the suffix heads of the words $U_{p+1}$ and $U_{q+1}$ are graphically equal as they are equal to the same suffix head.

Case 2. In one of the words, for example, in $U_{q+1}$, the suffix head coincides with the prefix head. In other words, the subword $\gamma_{l} S`_1$ is not a prefix head. The same way as in the case 1, for the word $V_{q+k}$ we can get the suffix representation of its subword $S_1 \theta_r \ldots  \theta_1,$ where $\gamma_{l} S_1$ is not a defining word. However, this means that the sequence of transformations, starting from $V_{q+k}$, does not proceed due to the absence of a head in it.

Case 3. If in both $U_{p+1}$ and $U_{q+1}$ words the prefix coincides with the suffix head, the statement of the lemma follows from the unambiguity of suffix transformations and the direct forms of the representations of these words.

Lemma is proved.
  \end{proof}
Let us formulate another lemma, which easily follows from the just proved simple induction on the number of suffix transformations
 
  \begin{lemma}   \label{ekvivalent2} 
 If the terms $U_{i}$ and $U_{j}$ of an infinite sequence \eqref{pref.preob} have the same endings, then the terms $U_{i+t}$, $U_{j+t}$, obtained from $U_{i}$ and $U_{j}$ by suffix transformations, also have the same endings for any $t.$ 
 \end{lemma}

 \begin{definition}    \label{perewal} 
 Suppose \eqref{pref.preob} is a sequence of suffix transformations and
 $U_{p_{1}}, U_{p_{2}}, U_{p_{3}}, U_{p_{4}}, \ldots$, is an infinite sequence of the words from that sequence, so that 
\begin{equation} \label{osnowa} 
         U_{p_{2\lambda}} \eqcirc K_{p_{2\lambda}} A D_{p_{2\lambda}}  L_{p_{2\lambda}}
\end{equation}
\begin{equation}              \label{osnowa 1}                                                
         U_{p_{2\lambda+1}} \eqcirc K_{p_{2\lambda+1}} C_{p_{2\lambda+1}} B L_{p_{2\lambda+1}} 
\end{equation}
 where $K_{i}, L_{j}$ are the subwords of the spans, $A$ and $B$ are the prefix heads of the words \eqref{osnowa} and \eqref{osnowa 1} respectively, that, as a result of suffix transformations $A \rightarrow D_{1} D_{2}$, or $B \rightarrow H_{1} H_{2}$ of the semigroup $\Pi$, the subwords $D_{1}$ and $H_{2}$ join the spans of the words $U_{p_{2\lambda}}$ and $U_{p_{2\lambda+1}}$ respectively.
$H_{1}$ is constant for the entire subsequence and has the same suffix guide letter. Then the words $U_{p_{2\lambda}}$ and $U_{p_{2\lambda+1}}$ will be called, respectively, left and right passes of the sequence \eqref{pref.preob}, given by the pair $(A, B)$ of the prefix heads $A, B.$
\end{definition}

The existence of a subsequence of passes follows from the fact that the beginning and ending of any word can be affected by the defining relations only a finite number of times, and the number of the defining words themselves with various representations is finite. In an infinite \eqref{pref.preob} sequence there will always be a subsequence with the specified properties for some pairs of words of the defining words $(A, B)$ of the semigroup $\Pi$.

Let 
\begin{equation}             
\label{pereval}   
U_{p_{1}}, U_{p_{2}}, U_{p_{3}}, U_{p_{4}}, \ldots,
\end{equation}
be a subsequence of passes of the sequence\eqref{pref.preob}, given by the pair of words $(A, B)$, with a fixed presentation for the entire subsequence, with all $B$ words having the same suffix guide letter.

Let the prefix head $B$ be translated into the word $H_1 H_2$ in the words $U_{p_{2}}$ and $U_{p_{4}}$ by the corresponding prefix transformations. By virtue of Lemma \eqref{ekvivalent} we cane replace the subsequences \eqref{pref.preob} by the subsequences of suffix transformations.
Further, the subsequences \eqref {pref.preob}, transforming the right passes $ U_ {p_ {2k + 1}} $ into their prefix forms, into the passes $ U_ {p_ {2k + 2}}$ respectively, and $ U_{p_{2k + 1}}$ in $ U_ {p_ {2k + 2}} $ are replaced by subsequences of suffix transformations, which, as was found in the Lemma \ref{ekvivalent}, is possible.
 \begin{lemma}   \label{qadro}
  For any $t$ and for any $k_{1}, k_{2}$ the words $U_{p_{2k_{1}} + t}$  and  $U_{p_{2k_{2}} + t}$  from the sequence   $U_{p_{2k}} \rightarrow \ldots \rightarrow U_{p_{2k+2}} \rightarrow \ldots \rightarrow $  of suffix transformations, that transform right passes of the sequence \eqref{pref.preob} into each other, have the same endings.
\end{lemma}

\begin{proof}
From the definition of the right pass for the pair ($A$, $B$) it immediately follows that the endings of the words $U_{p_{2k}}$ are the same for all $k.$
But then, by Lemma \ref{ekvivalent1} the endings of the words $U_{p_{2k_{1}} + 1}$  and  $U_{p_{2k_{2}} + 1}$ are also the same. Further, from Lemma
\ref{ekvivalent2} implies the correctness of the lemma being proved.
\end{proof}

Let us consider a subsequence of passes
$$U_{p_{1}}, U_{p_{2}}, U_{p_{3}}, U_{p_{4}},\ldots,$$ 
of a sequence \eqref{pref.preob} of prefix transformations defined for a pair ($A, B$) of the defining words. Without loss of generality, we can assume that the terms of this subsequence have the form
  \begin{equation}                   \label{perio1}
U_{p_{1}} \eqcirc K_{1} A B_{1} L_{1}
\end{equation}
\begin{equation}                   \label{perio2}
U_{p_{2\lambda}} \eqcirc K_{2\lambda} A_{2\lambda} B L_{2\lambda}, 
\end{equation}
\begin{equation}                   \label{perio3}
U_{p_{2\lambda+1}} \eqcirc K_{2\lambda+1} A B_{2\lambda+1} L_{2\lambda+1},
\end{equation}
\begin{equation}                   \label{perio4}
U_{p_{2\lambda+2}} \eqcirc K_{2\lambda+2} A_{2\lambda+2} B L_{2\lambda+2}
\end{equation}   
$\lambda =1,2,\ldots$  $K_{i}$ and $L_{i}$ are the left and right sides of the spans of the corresponding transformation words \eqref{pref.preob}, $A$ and $B$ are the corresponding prefix heads of words that are the passes of the sequence \eqref{pref.preob}.

Suppose that the sequence \eqref{pref.preob} is replaced by the sequence \begin{equation}              \label{suf}
U_{p_{1}}  \rightarrow U_{p_{2}} \rightarrow \ldots \rightarrow  U_{p_{2\lambda}} \rightarrow \ldots \rightarrow U_{p_{2\lambda+1}} \rightarrow \ldots
\end{equation} of suffix transformations, which is allowed by Lemma \ref{ekvivalent}. In that case the subword $B$ is the end of the words $U_{p_{2\lambda}},\ldots, U_{p_{2\lambda+2m}} \ldots$  by definition, but by lemmas  \ref{ekvivalent1}, \ref{ekvivalent2} for any number $t$ the ends of
the words $U_{p_{2\lambda}+t}$ coincide for all $\lambda,$, and, therefore, the sequence of ends of the words
$U_l$ of the sequence of suffix transformations obtained from \eqref{pref.preob}  turns out to be
periodic. This means that the group of different ends of words of the sequence is finite and the length of each of them does not exceed the length of the maximum from the
words of a finite subsequence starting with the word $U_{p_{2\lambda}}$ and ending with the word
$U_{p_{2\lambda+2}}.$

\begin{lemma}
\label{odnozn}
From the sequence \eqref{pref.preob} an infinite subsequence $$U_{\sigma_{1}} \rightarrow U_{\sigma_{2}} \rightarrow \ldots \rightarrow U_{\sigma_{m}} \rightarrow \ldots$$ can be choose.
For which exists a word $W$, which is the kernel of each of the terms of $U_{\sigma_{i}}$.
\end{lemma}

\begin{proof}
Let \eqref{suf} be the sequence of suffix transformations constructed from \eqref{pref.preob} in the way
shown above. And let $W_{1},W_{2},\ldots,W_{m},\ldots$ be the sequence of the kernels of the
terms \eqref{suf}. Since the kernels of the terms of the sequence \eqref{suf} are the subwords of the
ends of these words, then they are also finite numbers and, therefore, some of them,
say $W$, is repeated infinitely.
Let this be an infinite sequence of the suffix transformations in the form
$$U_{\sigma_{1}} \rightarrow U_{\sigma_{2}} \rightarrow \ldots \rightarrow U_{\sigma_{m}} \rightarrow \ldots.$$ And finally we replace the resulting sequence with the sequence of prefix
transformations.
\end{proof}

\begin{theorem} \label{polugrupp}
In a semigroup without cycles, the divisibility problem is decidable.
\end{theorem}
\begin{proof}
Suppose that we are given two words $U$ and $V$ in semigroup $\Pi$ and we
need to know whether the word $U$ is divisible by $V$. Let us run the algorithm, which will
conduct prefix transformations from the word $U$ to the word $V$ starting with the guide
letter $V$, which is the first letter of the word $V,$ and we will follow this process.

If in the beginning of the algorithm the form of word $U$ is obtained, starting from
the first letter of the word $V$, then instead of $U$ we take the resulting form without
the first letter and for that again we run the construction algorithm of prefix
transformations in the word obtained from $V$ by erasing the first letter.
It is clear that if the repeated construction process of prefix transformations reduces
$V$ to an empty word, then U is divisible by $V.$ The constructing algorithm of prefix transformations can terminate at some step due to the impossibility of finding a suitable prefix, in that case the answer is no.
The process of constructing prefix transformations can be continued infinitely, but
then, by \eqref{odnozn}, the sequence of prefix transformations turns out to be periodic,
so the answer is negative again, $U$ is not divisible by $V.$
\end {proof}

Using the result from (\cite{4}) that if in the semigroup without cycles the problem of
divisibility is decidable, then in a group with the same relations (a group without
cycles) the word equality problem is decidable, we get the following result.

\begin{theorem} \label{grupp}
The word equality problem is decidable in any group without cycles.
\end{theorem}

\end{document}